\documentclass[12pt]{article}

\usepackage[english]{babel}
\usepackage{amssymb,amsmath, amsthm, amscd, latexsym, epsfig}
\usepackage[affil-it]{authblk}
\usepackage{colordvi}
\usepackage[latin1]{inputenc}
\usepackage[T1]{fontenc}
\usepackage{upgreek}
\usepackage{enumerate}
\usepackage{xypic}
\usepackage{tikz}
\usepackage{bbm}
\usepackage{amssymb}
\usepackage[all]{xy}
\xyoption{all}
\usepackage{mathptmx}
\usepackage{babel,textcomp}
\usepackage{verbatim}

\hbadness= \tolerance
\newcommand{\id}[1]{\mathcal{J}_{#1}}
\newcommand{\dualm}[1]{#1'}
\newcommand{\obar}[1]{\overline{#1}}
\newcommand{\pwrset}[1]{2^{#1}}
\newcommand{\field}[1]{\mathbbm{#1}}
\newcommand{\polr}[3]{\field{#1}[#2_1,\ldots,#2_{#3}]}
\newcommand{\hide}[1]{}

\newcommand{\tensorp}[2]{#1\otimes#2}

\newcommand{\F}{\mathcal{F}}

\newcommand{\maxim}[1]{\mathbf{#1}}

\newcommand{\bfid}[1]{\mathcal{F}\big(#1\big)}
\newcommand{\fid}[1]{\mathcal{F}(#1)}
\newcommand{\sfid}[1]{S\mathcal{F}(#1)}
\newcommand{\pgen}[2]{\mathbf{#1}^{#2}}
\newcommand{\sr}[1]{\field{k}[#1]}

\newcommand{\nnn}[1]{\mathbb{N}_{0}^{n}}
\newcommand{\nn}[1]{\mathbb{N}_{0}}
\newcommand{\n}[1]{\mathbb{N}}

\theoremstyle{plain}
\newtheorem{theorem}{Theorem}[section]
\newtheorem{lemma}[theorem]{Lemma}
\newtheorem{proposition}{Proposition}
\newtheorem{corollary}{Corollary}

\theoremstyle{definition}
\newtheorem{definition}{Definition}[section]

\theoremstyle{remark}
\newtheorem*{rem}{Remark}

\DeclareMathOperator{\rk}{rk}
\DeclareMathOperator{\im}{im}
\DeclareMathOperator{\tor}{Tor}

\begin{document}

\author{Trygve Johnsen, Jan Roksvold\thanks{Corresponding author. E-mail address: \texttt{jan.n.roksvold@uit.no}}, Hugues Verdure}
\affil{Department of Mathematics, University of Troms\o,\\ N-9037 Troms\o, Norway}
\title{Betti numbers associated to the facet ideal of a matroid\footnote{The original publication is available at
http://link.springer.com/article/10.1007/s00574-014-0071-9}}
\author{Trygve Johnsen \and Jan Roksvold\and Hugues Verdure}
\maketitle
\begin{abstract}
\noindent To a matroid $M$ with $n$ edges, we associate the so-called facet ideal $\fid{M}\subset \polr{k}{x}{n}$, generated by monomials corresponding to bases of $M$. We show that when $M$ is a graph, the Betti numbers related to an $\nn{N}$-graded minimal free resolution of $\fid{M}$ are determined by the Betti numbers related to the blocks of $M.$ Similarly, we show that the higher weight hierarchy of $M$ is determined by the weight hierarchies of the blocks, as well. Drawing on these results, we show that when $M$ is the cycle matroid of a cactus graph, the Betti numbers determine the higher weight hierarchy -- and vice versa. Finally, we demonstrate by way of counterexamples that this fails to hold for outerplanar graphs in general.    
\end{abstract}

\section{Introduction}
By \emph{matroid} we shall, throughout, be referring to a finite matroid. So let $M=\big(E(M),\mathcal{I}(M)\big)$ be a matroid, with edge set and set of independent sets $E(M)$ and $\mathcal{I}(M)$, respectively. We denote the set of bases $\mathcal{B}(M)$. Whenever $\sigma\subset E(M)$, then $\big(\sigma,\{I\cap\sigma:I\in\mathcal{I}(M)\}\big)$ is of course itself a matroid. We shall denote this matroid simply as $\sigma$ as well. In other words, when dealing with a subset of $E(M)$, we shall throughout be considering it \textit{as a submatroid}. 

Several of the invariants associated to a matroid are found to be natural generalizations of corresponding invariants for codes, graphs or simplicial complexes. It is natural to study the interplay between such invariants, and how invariants of substructures determine the corresponding invariants of the ``global'' structure. One such set of invariants is the \textit{higher weight} hierarchy \[d_i(M)=\min\{|\tau|:\tau\subset E(M), |\tau|-\rk(\tau) = i\},\] where $\rk(\sigma)$ denotes denotes the rank of $\sigma$. (That is: the cardinality of its largest independent subset.) Note that if $M$ is the vectorial matroid derived from the parity check matrix of a linear code, then the higher weights of $M$ are equal to the higher Hamming-weights of the code. 

Another set of invariants is the so-called Betti numbers, whose algebraic nature requires us to establish a certain terminology. So, let $S=\polr{k}{x}{n}$ be the polynomial ring in $n$ variables over the field $\field{k}$, and let $\maxim{m}=\langle x_1,x_2,\cdots,x_n\rangle$. A complex \[\mathbf{X}:\begin{CD}
\cdots@<<<X_{i-1}@<\phi_{i}<<X_{i}@<<<\cdots                                   
                                   \end{CD}
\]over $S$ is said to be \textit{minimal} whenever $\im\phi_i\subset \maxim{m} X_{i-1}$ for each $i.$  

A \textit{minimal (ungraded) free resolution} of an $S$-module $N,$ is a minimal left complex
\[\begin{CD}
0@<<<F_{0}@<\phi_1<<F_{1}@<\phi_2<<F_{2}@<<<\cdots 
\end{CD}\]
where $F_i=S^{\beta_i}$ for some $\beta_i\in \nn{n}$, and which is exact everywhere except for in $F_0$, where $F_0/\im\phi_1\cong N.$ 

If $N$ is $\nn{N}$- or $\nnn{N}$-graded, we may form $\nn{N}$\textit{-} or $\nnn{N}$\textit{-graded minimal free resolutions}, in which case \[F_i=S(-r_1)^{\beta_{i,1}}\oplus S(-r_2)^{\beta_{i,2}}\oplus\cdots\oplus S(-r_l)^{\beta_{i,l}}\] for some integers $r_j,$ or \[F_i=\bigoplus_{\mathbf{a}\in\nnn{N}}S(-\mathbf{a})^{\beta_{i,\mathbf{a}}},\] respectively. In both of these latter cases we also require the boundary maps to be degree-preserving. The global Betti numbers $\{\beta_i\}$ of an ungraded resolution, the $\nn{N}$-graded Betti numbers $\{\beta_{i,j}\},$ and the $\nnn{N}$-graded Betti numbers $\{\beta_{i,\mathbf{a}}\}$ are all invariants of $N$, as any two (graded/ungraded) minimal free resolutions are isomorphic.
Choosing $N$ to be certain $S$-modules connected to the matroid $M,$ these Betti numbers become matroidal invariants as well. A frequently studied example is when $N$ is the so-called \textit{Stanley-Reisner ideal} $\mathcal{J}_M\subset S,$ generated by monomials corresponding to minimal non-faces (circuits) of the matroid. In \cite{JV}, by the first and third author, one clarifies the connection between higher weights and the Stanley-Reisner ideal.

Alternatively, one might study the \textit{facet ideal} $\mathcal{F}(M)$ of $S$, generated by monomials corresponding to bases of $M$. This ideal is investigated in e.g.~\cite{Far}. In this paper we shall be inspired by graphic matroids and ($\nn{N}$- and ungraded) minimal free resolutions of their facet ideals. Generalizing the concepts of $2$\textit{-connected} and a \textit{block}, familiar from the theory of graphs, we find that the $\nn{N}$-graded Betti numbers of a matroid are determined by the $\nn{N}$-graded Betti numbers of each of its blocks. This is done in Section \ref{cg}, where we give a concrete and easy method for computing the Betti numbers of any matroid given the Betti numbers of each of its blocks. 

A straightforward proof of the fact that $\F(M)$ is actually the Stanley-Reisner ideal of the Alexander dual of the matroid dual of $M$ is found in Section \ref{two}, for the benefit of the reader. As a result of this connection, minimal resolutions of facet ideals of matroids (from now on: \textit{matroidal} facet ideals) are particularly simple.

The Betti numbers of the facet ideal always give full information about the face numbers of the dual matroid $\dualm{M}$,
and therefore the first Hamming weight $d_1$ of $\dualm{M}$ (See Remark \ref{mindist} below). From a coding-theoretical point of view, this is in itself a reason for being interested in Betti numbers of a matroidal facet ideal; for whenever  $\dualm{M}$ corresponds to linear dependence amongst columns of a generator matrix for some code, the Betti numbers thus determine the code's minimum distance.

Complementing the result obtained in Section \ref{cg}, we demonstrate in Section \ref{H-weights} that the higher weights of a matroid are also determined by, and easily computed from, the higher weights of each of its blocks. 

A natural and clearly related question is whether the Betti numbers of a matroidal facet ideal determine the higher weight hierarchy of the matroid. As can be seen in e.g.~\cite{JV}, this is not true in general. One could however imagine that they do so for particularly well-behaved subclasses. Indeed, as an application of our main result, we show in Section \ref{cactus} that for graphic matroids stemming from cactus graphs, which are outerplanar, the higher weight hierarchy and the ordered set of $\nn{N}$-graded Betti numbers associated to the facet ideal do in fact determine each other.

In Section \ref{og}, we demonstrate, by way of counterexamples, that this fails to be the case for outerplanar graphs in general. This is an indication of how far the Betti numbers are from determining the full weight hierarchy in general. 

\section{The matroidal facet ideal}\label{two}
In this section we define the facet ideal of a simplicial complex, and identify it as the Stanley-Reisner ideal of another simplicial complex -- arising from the original one through a sequence of duality operations. This, in turn, implies that a matroidal facet ideal has so called \textit{linear} resolution over any field.

Let $\field{k}$ denote a field, and let $\Delta$ and $M$ be an (abstract) simplicial complex and a matroid, respectively, both on $[n]=\{1,\ldots,n\}$. (Recall that every matroid is also a simplicial complex.) For $\tau\subset [n],$ let $\mathbf{x}^{\tau}$ denote the square-free monomial in $\polr{k}{x}{n}$ that contains the factor $x_i$ if and only if $i \in \tau$. The \textit{Stanley-Reisner ideal} of $\Delta$ is the (square-free) monomial ideal \[\id{\Delta}=\langle \mathbf{x}^{\tau}:\tau\notin \Delta\rangle.\]

More particular to our studies shall be the following ideal, also treated in e.g.~\cite{Far}:
\begin{definition}
The \textit{facet ideal} of $\Delta$, is \[\fid{\Delta}=\langle \mathbf{x}^{\sigma}:\sigma\text{ is a facet of }\Delta\rangle.\] 
\end{definition}
Note that both the Stanley-Reisner ideal and the facet ideal are square-free and monomial, and that in the case of a matroid, the generators of the facet ideal correspond to bases of the matroid.

\begin{definition}

The \textit{Alexander dual} $\Delta ^*$ of $\Delta$, is \[\Delta ^*=\{\obar{\tau}\in[n]:\tau\notin\Delta\},\] while the \textit{dual matroid} $\dualm{M}$ of $M$ is \[\mathcal{B}(\dualm{M})=\{\overline{\beta}: \beta\in \mathcal{B}(M)\},\] where $\overline{\beta} = [n]\smallsetminus\beta$.

\end{definition}
\begin{proposition} \label{Alexander}
Let $M$ be a matroid, then $\fid{M}=\id{(\dualm{M})^*}.$
\end{proposition}
\begin{proof}
By definition, we have 
\begin{eqnarray*}(\dualm{M})^*&=&\{\obar{\beta}:\beta \not\in \mathcal{J}(\dualm{M})\},\text{ which is equal to} \\
&=&\pwrset{[n]} \smallsetminus \{\obar{\mu}:\mu\in \mathcal{J}(\dualm{M})\}.\end{eqnarray*} 
The Stanley-Reisner ideal of $(\dualm{M})^*$ then, is \[\id{(\dualm{M})^*}=\langle \mathbf{x}^{\obar{\mu}}:\mu\in \mathcal{J}(\dualm{M})\rangle.\]

Note that \[\langle \mathbf{x}^{\obar{\mu}}:\mu\in \mathcal{J}(\dualm{M})\rangle \subset \langle \mathbf{x}^{\obar{\mu}}:\mu\in \mathcal{B}(\dualm{M})\rangle;\] for if $\mu\in \mathcal{J}(\dualm{M})$, then $\mu\subset \beta$ for some $\beta\in \mathcal{B}(\dualm{M})$, such that $\obar{\beta}\subset\obar{\mu}$ and $\mathbf{x}^{\obar{\mu}}\subset\langle \mathbf{x}^{\obar{\beta}}\rangle.$ 

Since clearly \[\langle \mathbf{x}^{\obar{\mu}}:\mu\in \mathcal{J}(\dualm{M})\rangle \supset \langle \mathbf{x}^{\obar{\mu}}:\mu\in \mathcal{B}(\dualm{M})\rangle,\] we thus have \[\begin{array}{rcl}\id{(\dualm{M})^*}&=&\langle \mathbf{x}^{\obar{\mu}}:\mu\in \mathcal{J}(\dualm{M})\rangle\\
&=&\langle \mathbf{x}^{\obar{\mu}}:\mu\in \mathcal{B}(\dualm{M})\rangle\\
&=&\langle \mathbf{x}^{\sigma}:\sigma\in \mathcal{B}(M)\rangle\\
&=&\fid{M}.\end{array}\]
\qed\end{proof}

\begin{lemma}\label{Shell}
The facet ideal of a matroid $M$ has \textit{linear} minimal $\nn{N}$-graded free resolution. That is, a minimal free resolution of the form \[                                                                                   
0 \leftarrow S\big(-r\big)^{n_0}\leftarrow S\big(-(r+1)\big)^{n_1}\leftarrow\cdots\leftarrow S\big(-(r+l)\big)^{n_l}\leftarrow0,
\] where $r=\rk{(M)}$ and $l=|E(M)|-\rk{(M)}.$
\end{lemma}
\begin{proof}
This follows from \cite[Theorem 4 and Proposition 7]{ER} in combination with Proposition \ref{Alexander}.
\qed\end{proof}
\begin{rem}  \label{mindist}
Let $f_i(\Delta)$ denote the number of faces of dimension $i$ of the simplicial complex $\Delta$. From \cite[formula (1)]{ER} and \cite[Theorem 4]{ER} it follows that the Betti numbers of the facet ideal of a matroid $M$, in virtue of being the Stanley-Reisner ideal of $(\dualm{M})^*$, determine the face numbers $f_{i}(\dualm{M})$ of the dual matroid $\dualm{M}$. Consequently, these Betti numbers determine $d_1(\dualm{M})$ as well, since \[d_1(\dualm{M})=\min\left\{|\tau|:\tau\subset E(M), |\tau|-\rk_{\dualm{M}}(\tau) = 1\right\}=\min\left\{i : f'_{i-1} \ne \binom{n}{i}\right\}.\] In particular, when $M$ is the vectorial matroid derived from the parity check matrix of a linear code $C$ we thus see that the Betti numbers associated to $M$  determine the minimum distance of the dual code $C^{\perp}$. Through Wei duality then, they also give \textit{some} information about the higher weights of $C$ itself -- see \cite{Wei}.\end{rem}

\section{Blocks and Betti numbers} \label{cg}
Since every graphic matroid is isomorphic to the cycle matroid of some connected graph, there is no real parallel for matroids to the notion of a $1$-connected graph. In order to describe a property of matroids similar to that of being $2$-connected (for graphs), one introduces the relation $\xi$ on $E(M),$ where $e\ \xi \ f$ if either $e=f$ or if there is some circuit containing both $e$ and $f$. For a proof that this constitutes an equivalence relation on $E(M)$ see \cite[Proposition 4.1.2]{Oxl}. The equivalence classes of $\xi$ are referred to as the (connected) components or \textit{blocks} of $M$. Whenever $E(M)$ is either empty or itself a block, $M$ is said to be \textit{connected}. 

Now let $S=\polr{k}{x}{n}.$ If $m\leq n$ and $I$ is an ideal in \[\polr{k}{x}{m}=S',\] we let $SI$ denote the $S$-ideal generated by the same generators as $I$. That is, if \[I=\langle g_1,\ldots, g_k\rangle\subset S',\] then \[SI=\{s_1g_1+\cdots+s_kg_k: s_i\in S\}.\] 

\begin{proposition}\label{fid}
Let $B_1,B_2,\dots,B_t$ be the blocks of a matroid $M$. Then \[\fid{M}=\big(\sfid{B_1}\big)\big(\sfid{B_2}\big)\cdots\big(\sfid{B_t}\big).\] 
\end{proposition}
\begin{proof}
Observe that both $\fid{M}$ and $\big(\sfid{B_1}\big)\big(\sfid{B_2}\big)\cdots\big(\sfid{B_t}\big)$ are square-free monomial ideals. Furthermore, the generating set defining each of these ideals are both minimal with respect to cardinality. It is well known that every monomial ideal has a \textit{unique} minimal set of monomial generators; see e.g.~\cite[p.~4, Lemma 1.2]{MS}. 

Let $\pgen{x}{\sigma}$ be a generator for $\fid{M}.$ In other words: Let $\sigma$ be a basis for $M.$ Then $\sigma\cap B_i$ does not contain a circuit, and is thus independent in $B_i$. Now assume that $B_i\neq\sigma\cap B_i,$ and let $e\in B_i\smallsetminus(\sigma\cap B_i).$ Since $\sigma$ is a basis, $\sigma\cup e$ will contain a circuit. Furthermore, since $B_i$ is an equivalence class, this circuit will be contained in $B_i.$ In other words, $\sigma\cap B_i$ is a basis for $B_i$. Similarly, if $B_i=\sigma\cap B_i$ then, since any block with more than two elements must contain a circuit, we necessarily have that $|B_i|=1$ and thus that $\sigma\cap B_i$ is a basis for $B_i.$  

Since $\sigma=\bigcup_{i=1}^{t}\sigma\cap B_i$, we conclude that \[\pgen{x}{\sigma}=\pgen{x}{\cup_{i=1}^{t}\sigma\cap B_i}=\prod_{i=1}^{t}\pgen{x}{\sigma\cap B_i}\in\big(\sfid{B_1}\big)\big(\sfid{B_2}\big)\cdots\big(\sfid{B_t}\big).\]

Conversely, let $\prod_{i=1}^{t}\pgen{x}{\tau_i}=\pgen{x}{\cup_{i=1}^{t}\tau_i}$ be a generator for $\big(\sfid{B_1}\big)\big(\sfid{B_2}\big)\cdots\big(\sfid{B_t}\big).$ Then $\bigcup_{i=1}^{t}\tau_i$ contains some basis $\sigma$ of $M$. For if $e\in E(M)\smallsetminus(\bigcup_{i=1}^{t}\tau_i)$, then $e\cup\tau_i$ contains a circuit for some $i$ -- which implies that $e\cup(\bigcup_{i=1}^{t}\tau_i)$ contains this circuit as well. Consequently, \[\pgen{x}{\cup_{i=1}^{t}\tau_i}\in \langle \pgen{x}{\sigma}\rangle\subset\fid{M},\] and this concludes our proof. 
\qed\end{proof}

Proposition \ref{fid} is key to the proof of Theorem \ref{Tigran}, stated below. We point out that if $m\leq n$ and $I\subset S'=\polr{k}{x}{m}$ is an ideal with minimal graded free resolution \[
\begin{CD}                                                                                   
0@<<<F_0@<\phi_1<<F_1@<<<\cdots@<\phi_l<<F_l@<<<0,
\end{CD}
\] where $F_i=\bigoplus_{j=1}^{n_i}S'(-r_{i,j}),$
then \[
\begin{CD}                                                                                   
0@<<<S\otimes_{S'}F_0@<1_S\otimes\phi_1<<S\otimes_{S'}F_1@<<<\cdots@<1_S\otimes\phi_l<<S\otimes_{S'}F_l@<<<0
\end{CD}
\]
is a minimal graded free resolution of the $S$-module $S\otimes_{S'}I$, with the same grading as the original one. 

Proof of the following proposition is deferred until the end of this section.
\begin{theorem}\label{Tigran}
Let $M$ be a matroid, and let $S=\polr{k}{x}{|E(M)|}$. Let $B_1,B_2,\ldots, B_t$ be the blocks of $M$. For each $1\leq i \leq t$, let \[0\leftarrow S\big(-r_i\big)^{n_0,i}\leftarrow S\big(-(r_i+1)\big)^{n_1,i}\leftarrow\cdots\leftarrow S\big(-(r_i+l_i)\big)^{n_{l_i},i}\leftarrow 0.\] be a (linear) $\nn{N}$-graded minimal free resolution of $S\fid{B_i}$. If \[l=l_1+l_2+\cdots+l_t,\] \[r=r_1+r_2+\cdots+r_t,\] and \[\beta_i=\sum_{u_1+u_2+\cdots+u_t=i}n_{u_1,1}n_{u_2,2}\cdots n_{u_t,t},\] then  \[ 0\leftarrow S\big(-r\big)^{\beta_0}\leftarrow S\big(-(r+1)\big)^{\beta_1}\leftarrow\cdots\leftarrow S\big(-(r+l)\big)^{\beta_l}\leftarrow 0\] is a  minimal free resolution of $\fid{M}$.
\end{theorem}
We shall make use of the following shorthand:
\begin{flalign*}\field{k}[X]:=&\polr{k}{x}{m},\\
\field{k}[Y]:=&\polr{k}{y}{n},\\
S:=&\field{k}[x_1,\ldots,x_m,y_1,\ldots,y_n].
\end{flalign*} 
Note that if $M$ is a $\sr{X}$-module and $N$ is a $\sr{Y}$-module, the $k$-algebra isomorphism \[S\cong\sr{X}\otimes_\field{k}\sr{Y}\] gives $M\otimes_\field{k}N$ the structure of an $S$-module through $(f\otimes g)(m\otimes n)=fm\otimes gn$.
\begin{lemma}\label{Andrei1}
Let $M$ be a $\sr{X}$-module, and let $N$ be a $\sr{Y}$-module. Then \[\big(S\otimes_{\sr{X}}M\big)\otimes_S\big(S\otimes_{\sr{Y}}N\big)\cong M\otimes_\field{k}N\] as $S$-modules.
\end{lemma}
\begin{proof}
\[\big(S\otimes_{\sr{X}}M\big)\otimes_S\big(S\otimes_{\sr{Y}}N\big)\cong M\otimes_{\sr{X}}\big(\sr{X}\otimes_{\field{k}}\sr{Y}\big)\otimes_{\sr{Y}}N \cong M\otimes_{\field{k}}N.\]\qed

\end{proof}
\begin{lemma}\label{Andrei2}
Under the same conditions as in Lemma \ref{Andrei1}:
\[\tor_0^{S}\big(S\otimes_{\sr{X}}M,S\otimes_{\sr{Y}}N\big)\cong M\otimes_\field{k}N,\] and 
\[\tor_i^{S}\big(S\otimes_{\sr{X}}M,S\otimes_{\sr{Y}}N\big)=0\] for $i\geq1.$ 
\end{lemma}
\begin{proof}
The first statement is immediate from Lemma \ref{Andrei1}. For the second statement, let \[\begin{CD}
0 @<<< P_0 @<<< P_1 @<<<\cdots @<<< P_l@<<< 0,                                                                                                  
                                                                                               \end{CD}\] 
be a projective $\sr{X}$-resolution of $M$. Since $S$ is free as a $\sr{X}$-module, the following is a projective $S$-resolution of $M\otimes_{\sr{X}}S$:
\[\minCDarrowwidth19pt
\begin{CD}
0 @<<< P_0\otimes_{\sr{X}}S @<<< P_1\otimes_{\sr{X}}S @<<<\cdots @<<< P_l\otimes_{\sr{X}}S@<<< 0.   
  \end{CD}
\] 
Tensoring with $S\otimes_{\sr{Y}}N$, we obtain the following complex over $\big(M\otimes_{\sr{X}}S\big)\otimes_{S}\big(S\otimes_{\sr{Y}}N\big):$
\[\minCDarrowwidth19pt\begin{CD}
0 @<<< \big(P_0\otimes_{\sr{X}}S\big)\otimes_{S}\big(S\otimes_{\sr{Y}}N\big) @<<< \big(P_1\otimes_{\sr{X}}S\big)\otimes_{S}\big(S\otimes_{\sr{Y}}N\big) @<<<\cdots\end{CD}\]
\[\minCDarrowwidth19pt\begin{CD} \ \ \ \ \ \ \cdots @<<< \big(P_l\otimes_{\sr{X}}S\big)\otimes_{S}\big(S\otimes_{\sr{Y}}N\big)@<<< 0.                                                               
                                                              \end{CD}
\] 
According to Lemma \ref{Andrei1}, this complex is isomorphic to \[\begin{CD}
0 @<<<P_0\otimes_{\field{k}}N @<<< P_1\otimes_{\field{k}}N @<<<\cdots @<<< P_l\otimes_{\field{k}}N@<<< 0,
                                                                       \end{CD}
\] which is a complex over $M\otimes_\field{k}N\cong\big(S\otimes_{\sr{X}}M\big)\otimes_S\big(S\otimes_{\sr{Y}}N\big)$. But $N$ is free as a $\field{k}$-module, so this latter sequence is exact (except for in $P_0\otimes_{\field{k}}N$).
\qed\end{proof}

Next, let 
\[\begin{CD}\mathcal{F}:0 @<<< F_0 @<\phi_1<< F_1 @<\phi_2<<\cdots @<\phi_r<< F_r @<<< 0                                                            
                                                                 \end{CD}\]
be a minimal free resolution of the $S$-module $S\otimes_{\sr{X}}M,$ and let 
\[
\begin{CD}
\mathcal{G}:0 @<<< G_0 @<\psi_1<< G_1 @<\psi_2<<\cdots @<\psi_s<< G_s @<<< 0,                                                                  
                                                                 \end{CD}\]
be a minimal free resolution of $S\otimes_{\sr{Y}}N.$ Extending the functor $(\bullet\otimes_S\bullet)$ to the translation category of complexes, as described in \cite{Nor}, we obtain a left complex $\mathcal{F}\otimes_S\mathcal{G}$ over $\big(S\otimes_{\sr{X}}M\big)\otimes\big(S\otimes_{\sr{Y}}N\big),$ for which, by definition: \[\big(\mathcal{F}\otimes_S\mathcal{G}\big)_i=\bigoplus_{u+v=i}F_u\otimes_S G_v,\] and whose boundary maps $\mathbf{\mathbf{d}}_i:\big(\mathcal{F}\otimes_S\mathcal{G}\big)_i\rightarrow\big(\mathcal{F}\otimes_S\mathcal{G}\big)_{i-1}$ are given by \[\mathbf{d}_i\Big(\begin{bmatrix}0\\\vdots\\c_{uv}\\\vdots\\0\end{bmatrix}\Big)=\big(\tensorp{\phi_u}{1_{G_v}}\big)(c_{uv})+(-1)^u\big(\tensorp{1_{F_u}}{\psi_v}\big)(c_{uv}).\]
\begin{lemma}\label{Levon}
The left complex \[
\minCDarrowwidth23pt\begin{CD}
0 @<<<\big(\mathcal{F}\otimes_S\mathcal{G}\big)_0 @<\mathbf{d}_1<< \big(\mathcal{F}\otimes_S\mathcal{G}\big)_1 @<\mathbf{d}_2<<\cdots @<\mathbf{d}_{r+s}<< \big(\mathcal{F}\otimes_S\mathcal{G}\big)_{r+s} @<<<0\end{CD}\] constitutes a minimal free resolution of the $S$-module
\[\big(S\otimes_{\sr{X}}M\big)\otimes_{S} \big(S\otimes_{\sr{Y}}N\big).\]
\end{lemma}
\begin{proof}
By definition of the torsion functor, as given in e.g.~\cite[p.~121]{Nor}, we have \[H_i(\mathcal{F}\otimes_S\mathcal{G})=\tor_i\Big(\big(S\otimes_{\sr{X}}M\big)\otimes\big(S\otimes_{\sr{Y}}N\big)\Big),\] which in combination with Lemma \ref{Andrei2} implies that our resolution is free. Minimality follows from minimality of $\mathcal{F}$ and $\mathcal{G}$.
\qed\end{proof}

The above \textit{``Künneth type''} result clearly extends, by way of induction, to any finite number of modules (of the specified kind).

\begin{corollary}\label{Iben}
In the above notation, let $S=\field{k}[X_1;X_2;\dots;X_t]$, and, for each $1\leq i\leq t,$ let $M_i$ denote a $\sr{X_i}$-module. If the $S$-module $S\otimes_{\sr{X_i}}M_i$ has minimal free resolution \[\begin{CD}
0@<<<F_{i,0}@<<<F_{i,1}@<<<\cdots@<<<F_{i,l_i}@<<<0,                                                                    
                                                                   \end{CD}
\] 
then the $S$-module \[\big(S\otimes_{\sr{X_1}}M_1\big)\otimes_S \big(S\otimes_{\sr{X_2}}M_2\big)\otimes_S \cdots\otimes_S \big(S\otimes_{\sr{X_t}}M_t\big)\] has minimal free resolution
\[\begin{CD}
0@<<<P_0@<<<P_1@<<<\cdots@<<<P_{l_1+l_2+\cdots+l_t}@<<<0,                                                                    
                                                                   \end{CD}
\]
where \[P_i=\bigoplus_{u_1+u_2+\cdots+u_t=i}\Big(F_{1,u_1}\otimes_S F_{2,u_2}\otimes_S \cdots\otimes_S F_{t,u_t}\Big).\] 
\end{corollary}

\begin{lemma}\label{Isak}
Let $I\subset \sr{X}$ and $J\subset \sr{Y}$ be ideals. Then \[\big(S\otimes_{\sr{X}}I\big)\otimes_{S}\big(S\otimes_{\sr{Y}}J\big)\cong\big(SI\big)\big(SJ\big)\] as $S$-modules. 
\end{lemma}
\begin{proof}
In light of Lemma \ref{Andrei1} it suffices to establish $\big(SI\big)\big(SJ\big)\cong I\otimes_{\field{k}}J,$  which is easily seen to be true. \qed\end{proof} 

We now have all we need to prove Theorem \ref{Tigran}. 

\begin{proof}[of Theorem \ref{Tigran}]
The result now follows from combining Lemma \ref{Isak} and Corollary \ref{Iben}, together with our initial observation that \[\fid{M}=\big(\sfid{B_1}\big)\big(\sfid{B_2}\big)\cdots\big(\sfid{B_t}\big).\] \qed
\end{proof}

\section{The higher weights}\label{H-weights}
Let $M$ be a matroid. In this section we shall draw on a result from \cite{JV} which implies that the higher weights of a matroid are determined by certain \textit{non-redundant} sets of cycles. It shall follow immediately from this that the higher weights of the blocks determine those of the matroid itself.

Recall that $C(M)$ denotes the set of circuits of $M$.	
\begin{definition}
A subset $\Sigma$ of $C(M)$ is said to be \textit{non-redundant} if for all $\mu\in\Sigma$ we have \[\bigcup_{\tau\in(\Sigma\smallsetminus\mu)}\tau\subsetneq\bigcup_{\tau\in\Sigma}\tau.\]
\end{definition}
Let $\sigma\subset E(M).$
\begin{definition}
The \textit{degree of non-redundancy} of $\sigma$, is 
\[\deg(\sigma)=\max\{n\in\nn\ :\tau_j\subset\sigma\text{ for }1\leq j\leq n \text{ and }\{\tau_1,\dots,\tau_n\} \text{ is non-redundant}\}.\] 
\end{definition}
 
\begin{lemma}\label{Hugues2}
\[|\sigma|-\rk(\sigma)=\deg(\sigma).\] 
\end{lemma}
\begin{proof}
This is \cite[Proposition 1]{JV}. 
\qed\end{proof}
\begin{lemma}\label{hvekter}
\[d_i(\sigma)=\min\{|\tau_1\cup\cdots\cup\tau_i|:\tau_j\subset\sigma \text{ for }1\leq j\leq i \text{ and }\{\tau_1,\dots,\tau_i\} \text{ is non-redundant}\}.\] 
\end{lemma}
\begin{proof}
Immediate from Lemma \ref{Hugues2}. 
\qed\end{proof}
\begin{proposition}
Let $B_1,\dots,B_t$ be the blocks of $M$. With the convention $d_0=0$, we have \[d_i(M)=\min\left\{\sum_{j=1}^{t}d_{k_j}(B_j):\sum_{j=1}^{t}k_j=i\right\}.\] 
\end{proposition}
\begin{proof}
By induction on the number $t$ of blocks; the induction step being an immediate consequence of Lemma \ref{hvekter}. 
\qed\end{proof}

\section{Cactus graphs} \label{cactus}
This section concerns a class of graphs normally referred to as cactus graphs or cacti. Applying the results obtained in Section \ref{cg}, we show that for cactus graphs with a known number of loops the set of higher weights and the ordered multiset of Betti numbers determine each other. As we shall see later on, this result does not extend to the superclass of outerplanar graphs.

\begin{definition}
A cactus graph is a finite, connected graph with the property that each block is either a cycle or a single edge. 
\end{definition}

Or equivalently: A finite, connected graph with the property that no pair of distinct cycles share an edge. Whenever $C_{1},C_{2},\ldots,C_{t}$ denote the cycles of a cactus graph, we let $n_i$ denote the length of $C_i$. We assume that $n_1\leq n_2\leq\cdots\leq n_t.$ 

A couple of initial observations: First, since the facet ideal of a graphic matroid has \textit{linear} $\nn{N}$-graded minimal free resolution over any field, the ungraded and $\nn{N}$-graded minimal free resolutions of $\bfid{M(G)}$ have the same Betti numbers. We shall therefore consider only ungraded minimal free resolutions throughout this section.

Secondly, observe that if $C_m$ is a cycle of length $m$, and $E$ is a graph containing only one edge (possibly a loop), then $\bfid{M(C_m)}$ has minimal (ungraded) free resolution \[\begin{CD}0@<<<S^{m}@<<<S^{m-1}@<<<0,\end{CD}                                                                                                              \] while $\bfid{M(E)}$ has minimal free resolution \[\begin{CD}0@<<<S@<<<0.\end{CD}\] In combination with Theorem \ref{Tigran} it follows that the minimal free resolution of $\bfid{M(C_m)}$ is equal to the minimal free resolution of $\bfid{M(C_m\cup E)}.$ This, in turn, implies that if $G$ is a cactus graph whose cycles are $C_{1},C_{2},\ldots,C_{t},$  then the minimal free resolution of $\bfid{M(G)}$ is equal to the minimal free resolution of $\bfid{M(C_{1}\cup C_{2}\cup\cdots\cup C_{t})}.$ In other words: the one-edge blocks have no impact upon the Betti numbers of a cactus graph. This fact shall eventually, in combination with Theorem \ref{Tigran}, enable us to demonstrate 
that 
for a cactus graph $G$, the global Betti 
numbers of a minimal free resolution of $\bfid{M(G)}$ determine the higher weights $\{d_i\}$ of $M(G).$ Note that the converse of this is rather trivial since for cactus graphs we have \[d_i=\sum_{j=1}^in_j,\] which implies that the higher weights determine the lengths $n_1,n_2,\ldots,n_t$ of the cycles of $G$ -- and therefore also the global Betti numbers of $\bfid{M(G)}$ (according to the above remarks). 

Note also that, with $|E_G|=n$, the $S$-ideal $\bfid{M(G)}$ has a natural $\nnn{N}$-grading -- and thus also an $\nnn{N}$-graded minimal free resolution \[\begin{CD}
0@<<<F_0@<<<F_1@<<<\cdots@<<<F_l@<<<0,                                                                                                                                                                                                                                                             \end{CD}
\] where $F_i=\bigoplus_{\mathbf{a}\in\nnn{N}}S(-\mathbf{a})^{\beta_{i,\mathbf{a}}}.$ In that case, we clearly have \[\beta_{0,\sigma}=
	\begin{cases}
		1, & \text{if $\sigma$ is a basis of $M(G)$} \\
		0, & \text{elsewise,} \\
	\end{cases}\]  
which implies that the $\nnn{N}$-graded Betti numbers of \textit{any graph} determine \textit{not only} the higher weights, but the matroid $M(G)$ in its entirety. 

We now return to the ungraded case.
\begin{theorem}\label{Tal}
Let $G$ be a cactus graph containing $t\geq1$ cycles $C_{1},C_{2},\ldots,C_{t},$ with $C_{i}$ of length $n_i,$ and let $S=\polr{k}{x}{|E_G|}.$ Let $\sigma_i$ denote the $i$-th elementary symmetrical polynomial in the $n_1,\ldots,n_t,$ that is:
\begin{flalign*}
&\sigma_0=1\\
&\sigma_1=n_1+n_2+\cdots+n_t\\
&\ \ \  \ \vdots\\
&\sigma_j=\sum_{1\le k_1 < k_2 < \ldots < k_j \le t} n_{k_1} \ldots n_{k_j}\\
&\ \ \  \ \vdots\\
&\sigma_t=n_1n_2\cdots n_t.\\ 
\end{flalign*}
Then the facet ideal of $M(G)$ has ungraded minimal free resolution \[\begin{CD}
0@<<<S^{\beta_0}@<<<S^{\beta_1}@<<<\cdots@<<<S^{\beta_{t}}@<<<0,                                                                                                                      \end{CD}
\] 
where \[\beta_i=\sum_{j=0}^{i}(-1)^j\binom{t-j}{i-j}\sigma_{t-j}.\]
\end{theorem}
\begin{proof}Clearly, any block of $M(G)$ is either a single edge or a circuit. By the above comments then, the minimal free resolution of $\bfid{M(G)}$ is equal to the minimal free resolution of $\bfid{M(C_{1}\cup C_{2}\cup\cdots\cup C_{t})}.$ From Theorem \ref{Tigran} then, we see that $\bfid{M(G)}$ has minimal free resolution \[\begin{CD}
0@<<<S^{\beta_0}@<<<S^{\beta_1}@<<<\cdots@<<<S^{\beta_{t}}@<<<0,                                                                                                                      \end{CD}
\]
where \begin{equation}\label{lign}\beta_i=\sum_{\{\Sigma\subset\{1,2,\ldots,t\}:|\Sigma|=i\}}\big(\prod_{v\in \Sigma}(n_v-1)\prod_{v\notin \Sigma}n_v\big).\end{equation} 
This implies that for each $t-i\leq j\leq t,$ every possible monomial $(-1)^{t-j}n_{k_1}n_{k_2}\cdots n_{k_j}$ with $1\leq k_1<k_2<\cdots<k_j\leq t$ is a summand of $\beta_i$ considered as a monomial in $n_1,n_2,\cdots,n_t$ \textit{and}, furthermore, that all these monomials occur the same number of times as summands. We infer that \[\beta_i=\sum_{j=0}^{i}(-1)^jc_{t-j}\sigma_{t-j},\] for some $c_{t-j}\in\mathbb{N}.$ 

In order to determine $c_{t-j}$, first observe that the number of $\Sigma\subset \{1,2,\ldots,t\}$ with $|\Sigma|=i$ is $\binom{t}{i}$. For each such $\Sigma$, the number of monomials in \[\prod_{v\in \Sigma}(n_v-1)\prod_{v\notin \Sigma}n_v\] of degree $t-j$ is $\binom{i}{i-j}.$ Since the number of terms in $\sigma_{t-j}$ is $\binom{t}{t-j},$ we conclude that the coefficient of $\sigma_{t-j}$ in $\beta_i$ is $(-1)^j\frac{\binom{i}{i-j}\binom{t}{i}}{\binom{t}{t-j}}=(-1)^j\binom{t-j}{i-j}.$   
\qed\end{proof} 
\begin{theorem}
The higher weight hierarchy $\{d_i\}$ associated to the cycle matroid of a loop-free cactus graph $G$ is determined by the Betti numbers of the ungraded minimal free resolution of the facet ideal $\bfid{M(G)}$ of $G$. 
\end{theorem}
\begin{proof}
Recall that, by assumption, we have $n_1\leq n_2\leq\cdots\leq n_t$. The identity $d_i=\sum_{j=1}^in_j$, valid for cactus graphs, clearly implies that the lengths $n_1,n_2,\ldots,n_t$ determine the higher weights. It will therefore suffice to show that the Betti numbers determine the multiset $\{n_j\}.$ 

It is immediately clear from (\ref{lign}) that the number $t$ of cycles of $G$ is determined by the Betti numbers, seeing as it is equal to the length of the minimal free resolution. Furthermore, we notice that for each $i,$ the coefficient of $\sigma_{t-i}$ in $\beta_i$ is $(-1)^i.$ In particular we have $\sigma_t=\beta_0,$ which implies that (knowing all the $\beta_i$s) the equation \[\sigma_{t-i}=(-1)^i\left(\beta_i-\sum_{j=0}^{i-1}(-1)^j\binom{t-j}{i-j}\sigma_{t-j}\right)\] enables us to obtain the remaining $\sigma_i$s recursively. 

Now, the fact that the polynomial \[X^t-\sigma_1X^{t-1}+\sigma_2X^{t-2}-\cdots+(-1)^t\sigma_t\] has the unique multiset of roots $\{n_1,n_2,\ldots,n_t\}$ implies that if $H$ is a cactus graph containing cycles of length $m_1,m_2,\ldots,m_s$, and if the Betti numbers of $H$ are equal to those of $G,$ then certainly $s=t$ and \[\{n_1,n_2,\ldots,n_t\}=\{m_1,m_2,\ldots,m_t\}\] as multisets, which was what we needed to prove. 
\qed\end{proof}

Note that if $G$ contains loops we no longer have that the number $t$ of cycles in $G$ is equal to the number of non-zero Betti numbers, and the above proof fails in that case. If, on the other hand, the number $l$ of loops is \textit{known}, then \[\beta_t=\beta_{t-1}=\cdots=\beta_{t-l+1}=0,\] and the proof goes through unchanged.

\begin{rem} \label{topological}
The cycle matroid of a single cycle of length $n$ is of course the uniform matroid $U(n-1,n)$ where a set of bases consists of all edge subsets of cardinality $n-1$. For a cactus graph with $t$ cycles of lengths $n_1,\cdots,n_t$
we see that there are $n_1n_2\cdots n_t$ spanning trees each consisting of (the set corresponding to) $n_t-1$ edges from each cycle, and in addition all edges not contained in any cycle. The edges not contained in any cycle have no significance for the global Betti numbers $\beta_i$, so for simplicity we disregard them. Hence we may view the cycle matroid of the cactus graph as the multi-uniform matroid
$U = U\big((n_1-1,n_1),\cdots,(n_t-1,n_t)\big),$ whose ground set is \[\big([n_1]\times\{1\}\big)\cup \big([n_2]\times\{2\}\big)\cup\cdots\cup\big([n_t]\times \{t\}\big)\] and whose independent sets are all the sets of the form \[\big(I_{n_1}\times\{1\}\big)\cup \big(I_{n_2}\times\{2\}\big)\cup\cdots\cup\big(I_{n_t}\times \{t\}\big),\] where $I_{n_i}$ denotes a subset of $[n_i]$ whose cardinality is less than or equal to $(n_i-1)$. The looked-for Betti numbers of this matroidal facet ideal can in principle be found by using Hochster's formula (which is valid over any field $\field{k}$):
\[\beta_{i,\sigma} = \tilde{h}_{|\sigma|-i-1}(V_{\sigma}),\]
where $V$ is the Alexander dual of the matroid dual of $U.$ 

We do not rule out that applying Hochster's formula in such a way might give an alternative proof of Theorem \ref{Tal}, but so far we have not been able to perform the necessary calculations.
\end{rem}

\section{Counterexamples for outerplanar graphs}\label{og}

As mentioned in the introduction, cactus graphs are special instances of \textit{outerplanar} graphs:
\begin{definition}
A finite graph is said to be \textit{outerplanar} if it has an embedding in the plane in which every vertex lies on the boundary of the outer face. 
\end{definition}
In this section we present counterexamples showing that for outerplanar graphs in general, the Betti numbers may fail to determine the higher weights -- and vice versa. Note that these counterexamples are the smallest ones possible (in terms of number of edges). 

First, consider\\
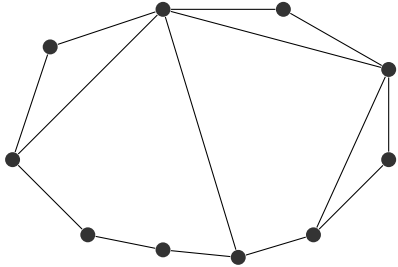
\begin{tabular}{ccc}
\begin{tikzpicture}[shorten >=1pt,->]
$G_1$  
\tikzstyle{vertex}=[circle,fill=black!80,minimum size=2pt,inner sep=2pt]
  \node[vertex] (G_1) at (-1,1){};
  \node[vertex] (G_2) at (-2.5,0.5){};
  \node[vertex] (G_3) at (-3,-1){};
  \node[vertex] (G_4) at (-2,-2){};
  \node[vertex] (G_5) at (-1,-2.2){};
  \node[vertex] (G_6) at (0,-2.3){};
  \node[vertex] (G_7) at (1,-2){};
  \node[vertex] (G_8) at (2,-1){}; 
  \node[vertex] (G_9) at (2,0.2){};
  \node[vertex] (G_10) at (.6,1){}; 
\draw (G_1) -- (G_2) -- (G_3) -- (G_4) -- (G_5) -- (G_6) -- (G_7) -- (G_8) -- (G_9) -- (G_10) -- (G_1) -- cycle;
\draw (G_1) -- (G_3) -- cycle;
\draw (G_1) -- (G_6) -- cycle;
\draw (G_1) -- (G_9) -- cycle;
\draw (G_9) -- (G_7) -- cycle;
\end{tikzpicture}
&
and
&
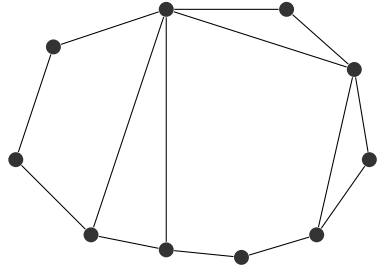
\begin{tikzpicture}[shorten >=1pt,->]
$G_2$  
\tikzstyle{vertex}=[circle,fill=black!80,minimum size=2pt,inner sep=2pt]
  \node[vertex] (G_1) at (-1.5,1){};
  \node[vertex] (G_2) at (-3,0.5){};
  \node[vertex] (G_3) at (-3.5,-1){};
  \node[vertex] (G_4) at (-2.5,-2){};
  \node[vertex] (G_5) at (-1.5,-2.2){};
  \node[vertex] (G_6) at (-0.5,-2.3){};
  \node[vertex] (G_7) at (.5,-2){};
  \node[vertex] (G_8) at (1.2,-1){}; 
  \node[vertex] (G_9) at (1,0.2){};
  \node[vertex] (G_10) at (.1,1){}; 
\draw (G_1) -- (G_2) -- (G_3) -- (G_4) -- (G_5) -- (G_6) -- (G_7) -- (G_8) -- (G_9) -- (G_10) -- (G_1) -- cycle;
\draw (G_1) -- (G_4) -- cycle;
\draw (G_1) -- (G_5) -- cycle;
\draw (G_1) -- (G_9) -- cycle;
\draw (G_9) -- (G_7) -- cycle;
\end{tikzpicture}
\end{tabular}\\
\noindent The ordered set of Betti numbers related to these graphs are equivalent since both facet ideals have $\nn{N}$-graded minimal free resolution \[\minCDarrowwidth11pt\begin{CD}0@<<< S(-9)^{393}@<<< S(-10)^{1459} @<<< S(-11)^{2187} @<<< \end{CD}\]
\[\minCDarrowwidth11pt\begin{CD}\; \; \; \; \;\;\;\;\; @<<<S(-12)^{1652}@<<< S(-13)^{628} @<<< S(-14)^{96} @<<< 0.\end{CD}\] Their respective weight hierarchies, however, are $\{3, 6, 8, 11, 14\}$ and $\{3, 6, 9, 11, 14\},$ which shows how the Betti numbers may fail to determine the higher weight hierarchy. In both cases $d_1=2$ for the dual matroid -- see Remark \ref{mindist}.

Next, consider 

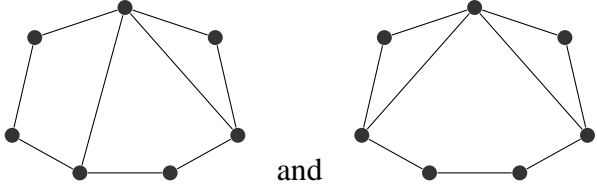
\begin{tabular}{ccc}\begin{tikzpicture}[shorten >=1pt,->]
$G_3$  
\tikzstyle{vertex}=[circle,fill=black!80,minimum size=2pt,inner sep=2pt]
  \node[vertex] (G_1) at (0,.7){};
  \node[vertex] (G_2) at (-1.2,0.3){};
  \node[vertex] (G_3) at (-1.5,-1){};
  \node[vertex] (G_4) at (-.6,-1.5){};
  \node[vertex] (G_5) at (.6,-1.5){};
  \node[vertex] (G_6) at (1.5,-1){};
  \node[vertex] (G_7) at (1.2,0.3){};
  \draw (G_1) -- (G_2) -- (G_3) -- (G_4) -- (G_5) -- (G_6) -- (G_7) -- (G_1) -- cycle;
  \draw (G_1) -- (G_4) -- cycle;
  \draw(G_1) -- (G_6) -- cycle;
\end{tikzpicture}
&
 and
&
\begin{tikzpicture}[shorten >=1pt,->]
$G_4$  
\tikzstyle{vertex}=[circle,fill=black!80,minimum size=2pt,inner sep=2pt]
  \node[vertex] (G_1) at (0,.7){};
  \node[vertex] (G_2) at (-1.2,0.3){};
  \node[vertex] (G_3) at (-1.5,-1){};
  \node[vertex] (G_4) at (-.6,-1.5){};
  \node[vertex] (G_5) at (.6,-1.5){};
  \node[vertex] (G_6) at (1.5,-1){};
  \node[vertex] (G_7) at (1.2,0.3){};
  \draw (G_1) -- (G_2) -- (G_3) -- (G_4) -- (G_5) -- (G_6) -- (G_7) -- (G_1) -- cycle;
  \draw (G_1) -- (G_3) -- cycle;
  \draw(G_1) -- (G_6) -- cycle;
\end{tikzpicture}
\end{tabular}

\noindent The graphic matroids related to these two outerplanar graphs have equivalent weight hierarchies, namely $\{3,6,9\}.$ However, the $\nn{N}$-graded minimal free resolutions \[0 \leftarrow S(-6)^{41} \leftarrow S(-7)^{92} \leftarrow S(-8)^{70} \leftarrow S(-9)^{18} \leftarrow 0\] and \[0 \leftarrow S(-6)^{39} \leftarrow S(-7)^{86} \leftarrow S(-8)^{64} \leftarrow S(-9)^{16} \leftarrow 0\] of $\bfid{M(G_3)}$ and $\bfid{M(G_4)}$, 
respectively, 
show that the higher weights fail to determine the Betti numbers for this particular pair.

\section{Acknowledgments}
The authors would like to thank Professor Andrei Prasolov, for his help with Lemmas \ref{Andrei1} and  \ref{Andrei2} and for taking an interest in our work. The first named author also thanks the organizers of the 12th ALGA Meeting IMPA, Rio de Janeiro, August 2012, for a most stimulating and rewarding conference.

\end{document}